\documentclass[12pt]{amsart}

\usepackage[top=1.25in, bottom=1.25in, left=1.05in, right=1.05in]{geometry}
\usepackage{amssymb}
\usepackage{graphicx}
\usepackage[export]{adjustbox}
\usepackage{floatrow}
\usepackage{parskip}
\usepackage[dvipdfm,colorlinks=true]{hyperref}

\usepackage[margin=10pt,font=small,labelfont=bf,belowskip=-15pt,aboveskip=0pt]{caption}
\DeclareCaptionFormat{empty}{}

\newtheorem{theorem}{Theorem}

\begin{document} 
\title{\textbf{New formulas and trees of Pythagorean triples}}
\author{Pavlo Deriy}
\email{pavel.deriy@gmail.com} 
\pagestyle{plain}
\begin{abstract}
An overview of known formulas and trees was made. New formulas explaining the construction of the Berggren tree are presented. The use of new formulas for certain problems, specifically the insolubility of the equation $x^4+y^4=z^4$, is shown. New ways of constructing non-classical trees, in particular a complete binary tree of Pythagorean triples, are presented. The new formulas are extended to other symmetric equations, in particular $x^n+y^n+z^n=0$.
\end{abstract}
\date{\today}
\maketitle
\thinmuskip=3mu \medmuskip=2mu \thickmuskip=1mu
\vspace*{-0.5cm}
\centering\textbf{1. Pythagorean trees. Retrospective}

\raggedright Primitive Pythagorean triples are pairwise coprime natural $x,y,z$ satisfying $x^2+y^2=z^2$. In what follows, we will consider only primitive triples. The number of such triples is infinite.

\begin{figure}[h]
\captionsetup{format=empty}
\includegraphics[height=9.1cm, right]{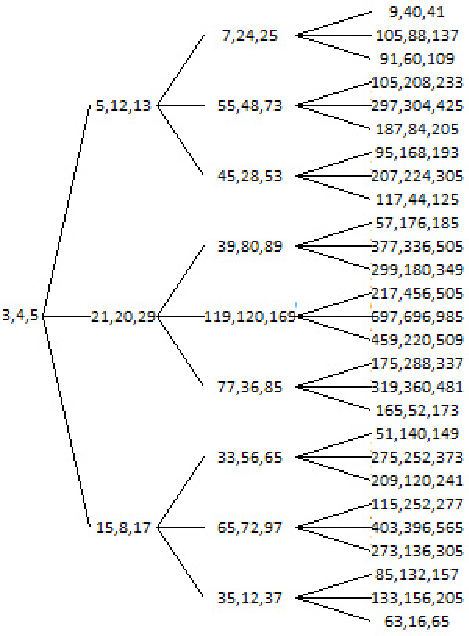}
\end{figure}
\begin{minipage}[c]{9.5cm}
\vspace*{-11.5cm}
In 1934, Berggren combined these triples into a ternary tree with the triple (3,4,5) at the root and proved it has all (and only) primitive Pythagorean triples without duplication [1].\\

Berggren found that the triple $(x+d,y+d,z+d)$ would also be Pythagorean if \mbox{$d=2(z-x-y)$}. He applied this shift to the triples $(-x,y,z),(-x,-y,z),(x,-y,z)$ and got three child triples. It is easy to see that such a shift is equivalent to multiplying $(x,y,z)$ by the matrices $A=\begin{bmatrix}1&$-$2&2\\2&$-$1&2\\2&$-$2&3\end{bmatrix}$, $B=\begin{bmatrix}1&2&2\\2&1&2\\2&2&3\end{bmatrix}$, $C=\begin{bmatrix}$-$1&2&2\\$-$2&1&2\\$-$2&2&3\end{bmatrix}$.\\

~\\The transition from child triples to the parent is carried out by matrices $A^{-1},B^{-1},C^{-1}$.
\end{minipage}

\vspace*{-1.2cm}There is only one path in the tree between the triples $(x,y,z)$ and $(x_1,y_1,z_1)$. You can go directly from the first triple to the second by multiplying the matrices along this path. So, the matrix $A^{-1}A^{-1}CAB$ leads from (7,24,25) to (275,252,373). This matrix is suitable only for those triples that are connected by exactly such a path.

In 2008, Price [2] and Firstov [3] built two more Pythagorean trees, also complete and unambiguous. These trees are also ternary, starting with (3,4,5) and generated by fixed sets of three matrices. Firstov proved that there are no other such trees [3].

~\\ \centering\textbf{2. Formulas of Pythagorean triples. Retrospective}

\raggedright Among relatively prime $x,y$, at least one is odd. We will consider $x$ to be odd. Then it follows from the equation $x^2+y^2=z^2$ that $y$ is even and $z$ is odd. 

There are several formulas for constructing Pythagorean triples. The most common and elegant option has been known since ancient times.

If $y=2y_1$ is even, then $y_1^2=(z+x)/2~\cdot~(z-x)/2$. The factors on the right-hand side of this equality are coprime since such are their sum and difference.

Therefore, these factors are equal to $u^2$ and $v^2$, so the following ratios are necessary 

(1)\qquad $x=u^2-v^2,y=2uv,z=u^2+v^2,u>v,(u,v)=1$.

The sufficiency of these ratios can be easily checked directly. 

Obviously, $u$ and $v$ have different parity, so $y$ is always divisible by 4. Either $u$ or $v$ can be even, making it difficult to enumerate the triples sequentially.

This drawback is easily eliminated if we express the triple in terms of divisors of the odd number $x$. In the equation $x^2=(z+y)(z-y)$ the factors on the right-hand side are coprime and equal to $a^2$ and $b^2$, whence 

(2)\qquad $x=ab,y=(a^2-b^2)/2,z=(a^2+b^2)/2,a>b$, $a,b$ are odd, $(a,b)=1$. 

Since $a$ and $b$ are odd, we can enumerate the triples by successively increasing $a$ and changing $b$ from 1 to $a$-2: $(a,b)$=(3,1);(5,1),5,3);(7,1),7,3),7,5);(9,1),(9,5),(9,7), and so on.

Both versions of the formulas go into each other by replacing $a=u+v,b=u-v$. In a specific task, one or the other option or even both at the same time may be more convenient. For example, it is clear from these formulas that every odd number $x=ab=u^2-v^2$ is the difference of two squares, where $u=(a+b)/2,v=(a-b)/2$. This fact is the basis of the Fermat method of factoring numbers. This also solves the problem of two pairs of squares with the same sum. Indeed, if $x=a_1b_1=a_2b_2$, then $x=u_1^2$-$v_1^2=u_2^2$-$v_2^2$, whence $u_1^2+v_2^2=u_2^2+v_1^2$. 

It follows from formulas (1),(2) that each Pythagorean triple is uniquely determined by factoring $x=ab$ or $y=2uv$.

In addition to these universal formulas, several other formulas are known for generating triples with special properties [4].

\newpage
\vspace*{-1.5cm}\centering\textbf{3. Conjugate Pythagorean triples}

\raggedright Formulas (1),(2) express triples either through divisors of an even number $y$ or through divisors of an odd number $x$. The new formulas use one divisor of an even number and one divisor of an odd number

(3)\quad $x_{1,2}=pq\pm q^2,y_{1,2}=pq\pm p^2/2,z_{1,2}=p^2/2+q^2\pm pq$, $p>q>p/2>0$, $p$ is even, $(p,q)=1$,

and also specify two conjugate triples at once.

\begin{theorem}
For every integer $p,q$, where $p$ is even, $(p,q)=1$, $p>q>p/2>0$, there exist two conjugate Pythagorean triples $(x_1,y_1,z_1)$ and $(x_2,y_2,z_2)$. These triples are expressed by formulas (3), and $(x_1,x_2)=q,(y_1,y_2)=p$.
\end{theorem}

\begin{proof}
If $x^2+y^2=z^2$, then $(x+y+z)^2=2(z+x)(z+y)$, and vice versa. All prime divisors of $z+y$ are divisors of $x+y+z$, and therefore are also divisors of $x$. Similarly, the divisors of $z+x$ are the divisors of $y$. Since $(x,y)=1$, then $2(z+x)$ and $(z+y)$ are coprime and equal to $p^2$ and $q^2$, where $p$ is even, $(p,q)=1$.  Considering the positivity of all numbers, we have $x+y+z=pq,x=pq-q^2,y=pq-p^2/2$, $z=pq-x-y=p^2/2+q^2-pq$.

Since $(p,q)=1$, the factors on the right-hand sides of $x=q(p-q),y=p(q-p/2)$ are coprime, and also $(x,y)=1$. If $p>q>p/2>0$, then $x>0$ and $y>0$, so $(x,y,z)$ is a primitive Pythagorean triple. In addition, $x_2=pq+q^2,y_2=pq+p^2/2,z_2=p^2/2+q^2+pq$ are also positive, pairwise coprime and satisfy the equation $x^2+y^2=z^2$. Finally, from $(p,q)=1$ it follows that $(p-q,p+q)=1$, and $(q-p/2,q+p/2)=1$, so $(x1,x2)=q,(y1,y2)=p$.
\end{proof}

\begin{theorem}
Every Pythagorean triple $(x,y,z)$ can be expressed in two ways, as
\end{theorem}
(4)\quad $x=p_1q_1-q_1^2,y=p_1q_1-p_1^2/2,z=p_1^2/2+q_1^2-p_1q_1$,	 \textit{and}\\
~~\qquad $x=p_2q_2+q_2^2,y=p_2q_2+p_2^2/2,z=p_2^2/2+q_2^2+p_2q_2$.

\begin{proof}
The first method is derived in Theorem 1. The second option is derived in the same way from the equation $(x+y-z)^2=2(z-x)(z-y)$.
\end{proof}

Formulas (3),(4) link Pythagorean triples into chains. If the triple $(x,y,z)$ is expressed in terms of $p_1,q_1$ and $p_2,q_2$, then the parameters $p_1,q_1$ form $(x^-,y^-,z^-)$ and $(x,y,z)$, and the parameters $p_2,q_2$ form $(x,y,z)$ and $(x^+,y^+,z^+)$.

\begin{theorem}
Each Pythagorean triple $(x,y,z)$ has four conjugate triples.
\end{theorem}

\begin{proof}
It follows from formulas (2) that each Pythagorean triple is uniquely determined by the decomposition of an odd number $x$ into coprime factors $x=ab$. According to formulas (4), $x=q_1(p_1-q_1)$ or $x=q_2(p_2+q_2)$, and here the factors are also coprime. Therefore, either $q_1=a,b$ and accordingly $p_1$ is found, or $q_2=a,b$ and $p_2$ is found. These options specify the same triple, but $x,y$ may be negative due to the violation of the condition $p>q>p/2>0$.
For each of these options, according to formulas (3), there is a conjugate triple.
\end{proof}

These four triples correspond to the parent and three child triples in the Berggren tree.  

For the triple (5,12,13) $x$=5=1$\cdot$5. Option $q_1=1,p_1-q_1=5$ leads to (5,-12,13) and conjugate (7,24,25). Option $q_1=5,p_1-q_1=1$ to (5,12,13) and (55,48,73). Option $q_2=1,p_2+q_2=5$ to (5,12,13) and (3,-4,5). Option $q_2=5,p_2+q_2=1$ to (5,-12,13) and (-45,-28,53). 

So, the new formulas explain the construction of the Berggren tree. They are also translated into formulas (2) by replacing $a=q,b=p-q$ (or $b=p+q$), but they are significantly different from the first options. Formulas (3) allow us to complete any coprime $p,q$ to elements $x,y$ for which $q|x,p|y$. The general formula for such $x,y$ can be obtained by replacing $p,q$ with $\alpha p,\beta q$, so $a=\beta q,b=\alpha p\pm \beta q$ (provided $a,b$ are odd and coprime). 

New formulas can be useful for some problems. For example, let's prove that the equation $x^4+y^4=z^4$ has no natural roots. Indeed, for such roots, there will be coprime $p,q$ such that $x^2=q(p-q),y^2=p(q-p/2)$. The factors on the right-hand sides of these equations are also coprime, so $q=a^2,p=b^2,p-q=b^2-a^2=c^2$. Hence, $b^2=a^2+c^2$, where the left part is divisible by 4, while the right part is only divisible by 2, which is impossible.~$\Box$

Additional possibilities are provided by the combination of new formulas with formulas (1),(2). Consider the problem of finding Pythagorean pairs such as $x=x_1x_2,y=y_1y_2$, wherein $x_1,y_1$ and $x_2,y_2$ are also Pythagorean pairs [5]. Let $x=q(p-q),y=p(q-p/2)$ according to formulas (3). Obviously, $q,p$ can easily be a Pythagorean pair or even a product of several Pythagorean pairs. So, it is necessary to find $q,p$ such that $p-q$ and $q-p/2$ form a Pythagorean pair. According to formulas (1), this requires $p-q=u^2-v^2,q-p/2=2uv$, from which $q=u^2-v^2+4uv,p=2u^2-2v^2+4uv$. So, if $u>v,2|uv,(u,v)=1$, then $p-q$ and $q-p/2$ form a Pythagorean pair. For $q,p$ to also be a Pythagorean pair, we need $q=a^2-b^2,p=2ab$. Thus, the problem is reduced to a system of equations $u^2-v^2+4uv=a^2-b^2,u^2-v^2+2uv=ab$, from which $2uv=a^2-b^2-ab$. But when $(a,b)=1$, the right-hand side of the last equation is always odd, so such Pythagorean pairs do not exist.~$\Box$

\centering\textbf{4. Modified formulas of triples}

\raggedright If formula (3-) expresses the triple $(x,y,z)$ through the parameters $p,q$, then formula (3+) expresses the child triple through the same $p,q$. But any triple can be expressed in terms of the parameters of any other triple. For example, consider the triples (7=7$\cdot$1,24,25) and (275=25$\cdot$11,252,373). According to formulas (2), for the first triple, 7=$ab$=7$\cdot$1, 24=($7^2$-$1^2$)/2, and for the second 275=$a_1b_1$=25$\cdot$11, 252=($25^2$-$11^2$)/2. One of the options for the transition from $a,b$ to $a_1,b_1$ will be $a_1=4a-3b,b_1=2a-3b$. So, the formulas

(2a)\qquad $x=(4a-3b)(2a-3b),2y=(4a-3b)^2-(2a-3b)^2,2z=(4a-3b)^2+(2a-3b)^2$,\\
~~\qquad~~\qquad $a>b$, $a,b$ are odd, $(a,b)=1$

will generate a Pythagorean triple whenever $a,b$ generate some other triple.

This modification can be done with other formulas as well. The modified variants will be as universal as the original ones if the mapping $(a,b)\to (a_1,b_1)$ is one-to-one and preserves relative primality, for example, $a_1=(a^2+b^2)/2, b_1=(a^2-b^2)/2$. Our mapping is not like that, since (3,1)$\to$(9,3) for it.

We will show how the modification of the formulas leads to the modification of the Berggren tree. For convenience, let's write the formulas of child triples in terms of formulas (2). Let $(x,y,z)$ be the parent triple, and $(x_{1-3},y_{1-3},z_{1-3})$ be the children. Omitting the details, we write that the following formulas hold for $x=ab$ $(a>b)$:

(5)~\quad $x_1=2b^2+ab,y_1=(a^2+3b^2)/2+2ab$;      $x_2=2a^2+ab,y_2=(3a^2+b^2)/2+2ab$;\\
~\qquad~\qquad $x_3=2a^2-ab,y_3=(3a^2+b^2)/2-2ab$.

\begin{figure}[h]
\captionsetup{format=empty}
\includegraphics[height=9.1cm, right]{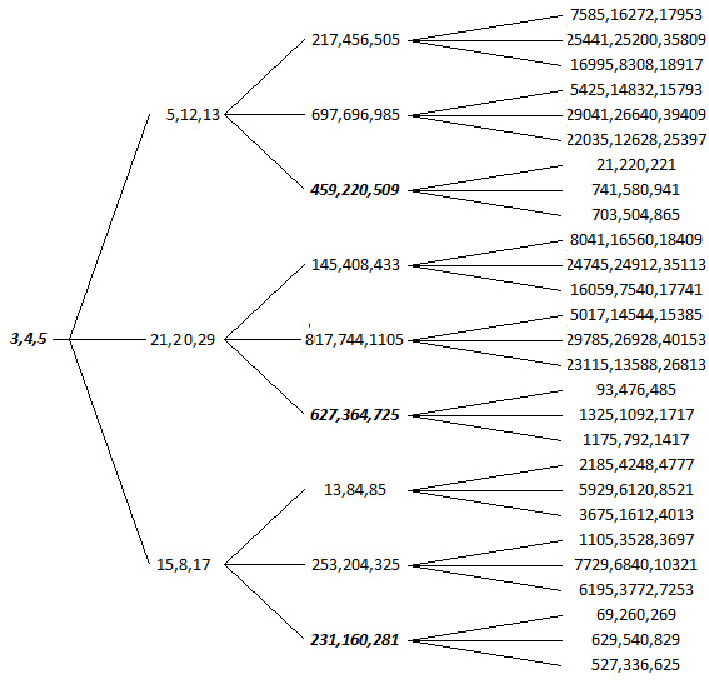}
\end{figure}

\begin{minipage}[c]{7cm}
\vspace*{-12.4cm}
Let's replace $a,b$ in formulas (5) with $(4a-3b),(2a-3b)$ and build a tree from the root (3,4,5). In those places where the values of the formulas are not coprime (highlighted in bold), the common factors were eliminated.\\

Palmer, Ahuja, Tikoo [6] gave formulas for replacing the transformation $(a,b)\to (a_1,b_1)$ with the 3x3 matrix.\\

However, for trees built according to modified formulas, the transition matrices from $(x,y,z)$ to $(x_{1-3},y_{1-3},z_{1-3)}$ are not fixed but different for each parent triple.
\end{minipage}

\vspace*{-1cm}
\centering\textbf{5. A new way of building trees}

\raggedright Slightly modifying Berggren's method, we note that the triple $(x+ad,y+bd,z+cd)$ will be Pythagorean if $a,b,c$ are arbitrary integers and $d=2(cz-ax-by)/(a^2+b^2-c^2)$ is integer.

To go from the parent triple to the children, we apply this transformation to the triples $(-x,y,z),(-x,-y,z),(x,-y,z)$, each time getting a new value of $d$. It is clear that when $a=b=c=1$ we will get the Berggren tree. The values $ad,bd,cd$ specify the shift from auxiliary triples to children. Applying this shift to the triple $(x,y,z)$, we get its parent triple $(-x_0,y_0,z_0),(-x_0,-y_0,z_0)$ or $(x_0,-y_0,z_0)$, depending on the branch of $(x,y,z)$.

\begin{figure}[t]
\captionsetup{format=empty}
\includegraphics[height=3.6cm, right]{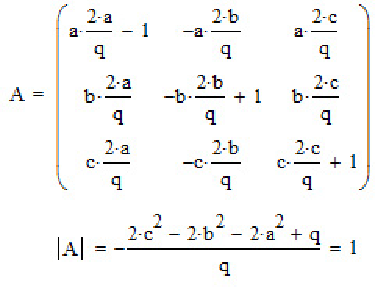}
\end{figure}

\begin{minipage}[t]{11.5cm}
\vspace*{-4.5cm}
For any $a,b,c$, it is possible to construct transformation matrices that will be fixed for all $x,y,z$.\\

Consider the transition $(-x,y,z)\to (x_1=-x+ad,y_1=y+bd,z_1=z+cd)$. For it, $d=(2ax-2by+2cz)/q$, where $q=a^2+b^2-c^2$. This transition can be performed by multiplying the matrix $A$ by the vector $(x,y,z)$. The determinant of this matrix is equal to 1.
\end{minipage}

\vspace*{-0.5cm}
The determinants of other matrices are also equal to $\pm 1$. The transition $(x,y,z)\to (x_0=x+ad,y_0=y+bd,z_0=z+cd)$ from any child triple to the parent triple is specified by matrix $D$.

When $q=\pm 1$ or $q=\pm 2$, then the matrices will be integer, and therefore unimodular. Triples with $a^2+b^2-c^2=1$ (for example, 4,7,8 or 41,137,143) were considered by Antalan and Tomenes [7]. Parameters $a=4,b=7,c=8$ correspond to matrices\\
$A=\begin{bmatrix}31&$-$56&64\\56&$-$97&112\\64&$-$112&129\end{bmatrix},B=\begin{bmatrix}31&56&64\\56&97&112\\64&112&129\end{bmatrix},C=\begin{bmatrix}$-$31&56&64\\$-$56&97&112\\$-$64&112&129\end{bmatrix},D=\begin{bmatrix}$-$31&$-$56&64\\$-$56&$-$97&112\\$-$64&$-$112&129\end{bmatrix}.$

\begin{figure}[h]
\captionsetup{format=empty}
\includegraphics[height=9.1cm, right]{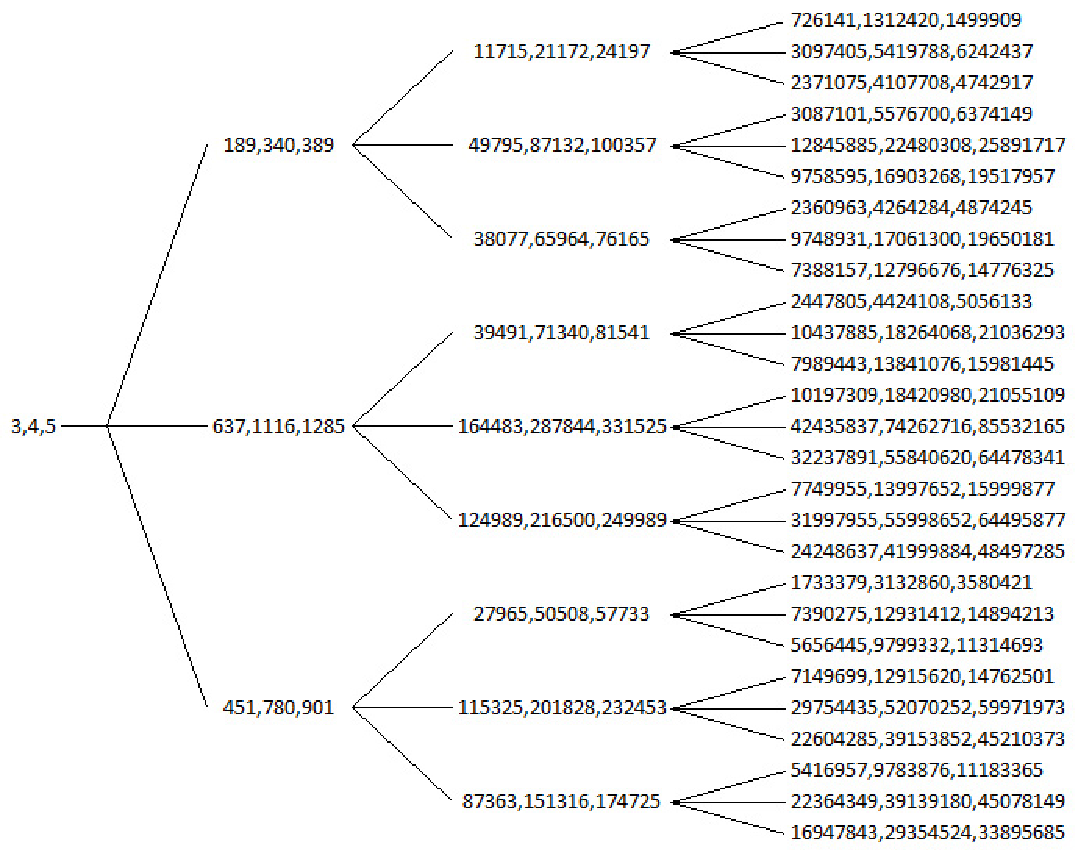}
\end{figure}

\begin{minipage}[t]{5.5cm}
\vspace*{-9cm}
The tree for $a=4,b=7,c=8$ with the root (3,4,5) is unambiguous since from each triple there is a unique backward path to the root.\\

~\\~\\~\\~\\But, obviously, this tree is incomplete.
\end{minipage}

The trees found by Price and Firstov cannot be constructed in this way, since they use different shifts for different triples.

\newpage
\centering\textbf{6. Exotic Pythagorean trees}

\raggedright The procedure described in the previous section can be slightly relaxed. First, if $d=p/q$ is not an integer, then the new triple must be multiplied by $q$. Second, common factors should be eliminated if they appear.

For each pair $(x,y,z)$ and $(x_1,y_1,z_1)$ from the equations $x_1$=$|x$+$ad|,y_1$=$|y$+$bd|,z_1$=$|z$+$cd|$ you can find the required $a,b,c$. But in the vast majority of cases, the resulting tree is incomplete or contains loops or "withered" branches. For example, (3,4,5)$\to$(57,176,185) when $a=6,b=18,c=19$, and at the same time (57,176,185)$\to$(3,4,5).

With $a=b=1,c=2$ a slightly modified Berggren tree is obtained. The chains of conjugate triples are the same here, but at each step, $x$ and $y$ are swapped.

\begin{figure}[h]
\captionsetup{format=empty}
\includegraphics[height=5.5cm, right]{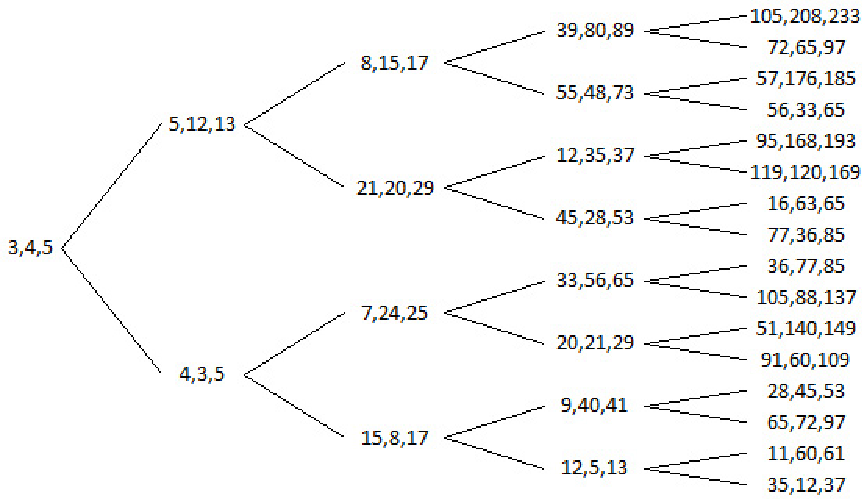}
\end{figure}

\begin{minipage}[t]{6.5cm}
\vspace*{-5cm}
There are more strange options. If, for $a=1,b=2,c=1$, the procedure is applied only to $(-x,-y,z)$ and $(x,-y,z)$, we get a complete binary tree where each triple occurs twice, as $(x,y,z)$ and $(y,x,z)$.
\end{minipage}

\vspace*{-0.8cm}
In this example, applying the procedure to $(-x,y,z)$ leads to degenerate branches with negative triples. When $a=3,b=4,c=3$, branches with negative triples also appear. After their exclusion, a complete tree remains, in which the nodes have both double and triple branching.

\centering\textbf{7. Formulas for equations of higher powers}

\raggedright The formulas (3) are derived from the equation $(x+y+z)^2=2(x+z)(y+z)$. For $n=3$, the formulas $x=pqr-p^3/3,y=pqr-q^3,z=pqr-r^3,p^3/3+q^3+r^3=2pqr$ are derived from the similar equation $(x+y+z)^3=3(y+z)(z+x)(x+y)$.

These formulas can be generalized to odd $n>2$. If in the equation $x^n+y^n+z^n=0$ you make a substitution $a=x+y+z,b=y+z,c=z$, from where $x=a-b,y=b-c,z=c$, then the equation will turn into $a^n-bM_n=0$. A polynomial $M_n$ is unique for every $n$. This equation shows that $a^n$ is divisible by $b=y+z$. Similarly, it is divisible by $z+x,x+y$, and since these binomials are pairwise coprime, then $(x+y+z)^n=S(y+z)(z+x)(x+y)$. If we transfer the common divisors from $S$ to the binomials, we get $(x+y+z)^n=S_1(uy+uz)(vz+vx)(wx+wy)$ with pairwise coprime factors on the right side. This is similar to the equation from which we derived formulas (3), so for some pairwise coprime $p,q,r,s$ $S_1=s^n,y+z=p^n/u,z+x=q^n/v,x+y=r^n/w,x+y+z=pqrs$. Hence $x=pqrs-p^n/u,y=pqrs-q^n/v,z=pqrs-r^n/w,p^n/u+q^n/v+r^n/w=2pqrs$. 

It is not difficult to prove that $u,v,w$ are divisors of $n$, so when $n$ is prime, the formulas are simplified. We emphasize that these formulas describe only the candidates for the roots of the equation. For real roots, the parameter $s$ is required to be consistent with $M_n$.

\vspace*{1cm}
\centering\textbf{8. Sockets}

\raggedright The equations from which the previous formulas were derived can be generalized as follows.

If every prime divisor of $a$ is a divisor of $b$, we will say that $a$ is \textit{included} in $b$ and write $a\to b$.  Let us denote by $\lambda$ the unordered set of integers $\{x_1,\ldots,x_m\}$, and by $\lambda_k$ this set without $x_k$.

\textsc{Definition}. Let $f(\alpha_1,\ldots,\alpha_{m-1})$ be a symmetric integer-valued function over the ring $Z$.  Then the set $\lambda$ is called ‘\textit{an $m$-order socket respectively to $f$}‘ if for any $i,j$ $(x_i,x_j)=1$, and $f(\lambda_i)\to x_i$. We also say that \textit{'$f$ generates $\lambda$'}.

A simple example of a socket is the set $\{3,5,22\}$ with the function $f=\alpha+\beta$. Indeed, $f(3,5)\to 22,f(3,22)\to 5,f(5,22)\to 3$.

Let's derive the relations for the socket elements.

We express $f$ in terms of elementary symmetric polynomials and construct a function $F$ with the same polynomials, but of the $m$th order. For example, if $f(\alpha,\beta)=(\alpha+\beta)^2-3\alpha\beta$, then $F(\alpha,\beta,\gamma)=(\alpha+\beta+\gamma)^2-3(\alpha\beta+\beta\gamma+\gamma\alpha)$, i.e., both functions are equal to $e_1^2-3e_2$.
 
Grouping in $F$ all the terms containing $x_k$, we get $F(\lambda)=x_kb_k+f(\lambda_k)$, where $b_k$ is the polynomial of $\lambda$. This shows that $f(\lambda_k)\to F(\lambda)$ and $F^n(\lambda)=Sf(\lambda_1)f(\lambda_2)\cdots f(\lambda_m)$ for some $S,n$. Let's transfer common divisors from $S$ to each $f(\lambda_i)$, then $F^n(\lambda)=S'u_1f(\lambda_1)u_2f(\lambda_2)\cdots u_mf(\lambda_m)$. On the right-hand side of this equation, all factors $u_if(\lambda_i)$ and $S'$ are pairwise coprime, so for some pairwise coprime $p_1,p_2,\ldots,p_m,s$, and associated $u_1,u_2,\ldots,u_m$ 

(6)\qquad $f(\lambda_k)=p_k^n/u_k,S=s^n\prod u_i,F(\lambda)=s\prod p_i,b_kx_k=s\prod p_i-p_k^n/u_k$.

Summing $F(\lambda)=x_kb_k+f(\lambda_k)$ over all $k$, we get $\sum p_i^n/u_i=c+(m-1)s\prod p_i$, where $c=F(\lambda)-\sum b_ix_i$.

Sockets have interesting properties that are beyond the scope of this article. They may be useful for some Diophantine problems. A significant limitation is the requirement for elements to be pairwise coprime. For the equation $x^2+y^2=c^2$ $\{x,y\}$ is a socket with $f=c-\alpha$, and for the equation $x^3+y^3+z^3=0$ $\{x,y,z\}$ is a socket with both $f=\alpha+\beta$ and $f=\alpha^2-\alpha\beta+\beta^2$, since $-x^3=(y+z)(y^2-yz+z^2)$. We emphasize once again that the formulas do not give the roots of the equations, but only the candidates for the roots.

\newpage
\centering\textbf{9. Theorem of Pythagoras}

\raggedright The sum of the squares of the legs of a right-angled triangle is equal to the square of the hypotenuse. There are many ways to prove this.

\begin{figure}[h]
\captionsetup{format=empty}
\includegraphics[height=4.5cm, right]{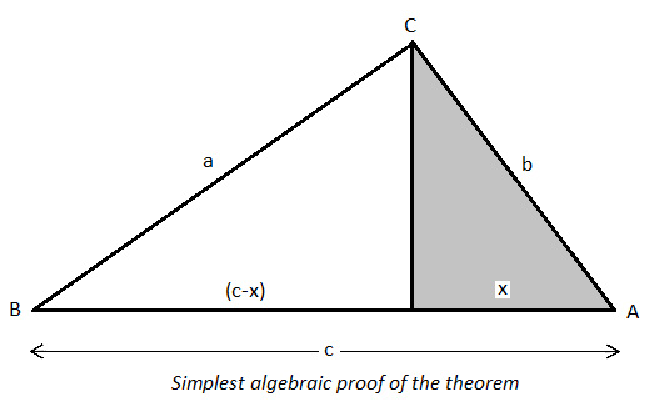}
\end{figure}

\begin{minipage}[t]{10cm}
\vspace*{-5cm}
"The simplest algebraic proof" [8] is based on this figure. The shaded triangle is similar to triangle $ABC$, so \mbox{$b:x=c:b$} and $b^2=cx$. Likewise, the unshaded triangle is similar to $ABC$, whence $a:(c-x)=c:a$, or $a^2=c^2-cx$. After summing, we have $a^2+b^2=c^2$.
\end{minipage}

\vspace*{-0.5cm}
An even simpler proof [9] also relies on these similar triangles. The areas of similar figures are related to each other, like squares of the corresponding linear dimensions (since the area is proportional to the product of two linear dimensions). Thus, the areas of our triangles are proportional to the squares of their hypotenuses $a,b,c$. That is, $s_a+s_b=s_c$, $s_a/s_c+s_b/s_c=a^2/c^2+b^2/c^2=1$, whence $a^2+b^2=c^2$.

\textsc{Corollary}. For two similar figures or surfaces, you can construct the same third one, whose area will be equal to the sum (difference) of the areas of the first two. The corresponding segments of these figures form a right-angled triangle.

In particular, a jacket with a sleeve length of 65 cm will take the same amount of fabric as two jackets with sleeves of 56 and 33 cm.


\begin{thebibliography}{10}
\bibitem{First} Berggren, B. "Pytagoreiska trianglar". Elementa: Tidskrift för elementär matematik, fysik och kemi, 17:129-139, 1934
\bibitem{Second} Price, H. Lee "The Pythagorean Tree: A New Species" \href{https://arxiv.org/abs/0809.4324}{arXiv:0809.4324}, 2008
\bibitem{Third} Firstov, V.E. "A Special Matrix Transformation Semigroup of Primitive Pairs and the Genealogy of Pythagorean Triples", Mathematical Notes, volume 84, number 2:263-279, 2008
\bibitem{Fourth} Weisstein, Eric W. \href{https://mathworld.wolfram.com/PythagoreanTriple.html}{“Pythagorean Triple”} From MathWorld—A Wolfram Web Resource
\bibitem{Fifth} Halbeisen, L., Hungerbühler, N. "Pairing Pythagorean Pairs" \href{https://arxiv.org/abs/2101.08163}{arXiv:2101.08163}, 2021
\bibitem{Sixth} Palmer, L., Ahuja, M., Tikoo, M. “Finding Pythagorean Triple Preserving Matrices”, Missouri Journal of Mathematical Sciences, volume 10, issue 2, Spring 1998, pages 99-105.
\bibitem{Seventh} Antalan, J.R.M. and Tomenes, M.D. A Note on Generating Almost Pythagorean Triples, International Journal Of Mathematics And Scientific Computing, Vol. 5, No. 2, 2015, pp. 100-102
\bibitem{Eighth} Gardner, M. "The Sixth Book of Mathematical Games from Scientific American", Chicago, IL: University of Chicago Press, pp. 155-157, 1984
\bibitem{Ninth} Penrose, R. "The Road to Reality", Jonathan Cape, London, pp. 31-32, 2004
\end{thebibliography}
\end{document}